\documentclass[11pt,reqno]{amsart}
\usepackage[a4paper,margin=28mm]{geometry}
\usepackage[T1]{fontenc}
\usepackage{lmodern,microtype}
\usepackage{amsmath,amssymb,mathtools}
\usepackage{cite}
\usepackage[hidelinks]{hyperref}
\numberwithin{equation}{section}
\newtheorem{theorem}{Theorem}[section]
\newtheorem{lemma}[theorem]{Lemma}
\newtheorem{proposition}[theorem]{Proposition}
\newtheorem{corollary}[theorem]{Corollary}
\theoremstyle{definition}

\newtheorem{example}[theorem]{Example}
\theoremstyle{remark}
\newtheorem{remark}[theorem]{Remark}
\newcommand{\Z}{\mathbb{Z}}
\newcommand{\Q}{\mathbb{Q}}
\newcommand{\C}{\mathbb{C}}
\allowdisplaybreaks[2]
\hypersetup{pdftitle={Hankel Determinants for Weighted Lattice Paths with Two Horizontal Steps},pdfauthor={Ying Wang and Zihao Zhang}}

\title[Hankel determinants for weighted lattice paths]{Hankel Determinants for Weighted Lattice Paths with Two Horizontal Steps}
\author[Y. Wang]{Ying Wang}
\address{School of Mathematics and Statistics, North China University of Water Resources and Electric Power, Zhengzhou 450045, PR China}
\email{wangying2019@ncwu.edu.cn}
\urladdr{\href{https://orcid.org/0009-0001-8547-0938}{https://orcid.org/0009-0001-8547-0938}}
\author[Z. Zhang]{Zihao Zhang}
\address{School of Mathematics and Statistics, Beijing Institute of Technology, 102400, P.R.~China}
\email{zihao-zhang@foxmail.com}
\urladdr{\href{https://orcid.org/0009-0001-4314-4456}{https://orcid.org/0009-0001-4314-4456}}
\thanks{Corresponding author: Zihao Zhang.}
\date{}

\begin{document}
\begin{abstract}
We evaluate all Hankel determinants of the formal power series
\(G=a/(1+b_1x+b_2x^2+cx^kG)\), where \(k\ge4\).
The defining equation arises from a weighted lattice-path model with horizontal steps of lengths one and two.
For \(k=4\), the determinants are expressed through Fibonacci polynomials.
For \(k\ge5\), explicit formulas describe each residue class of the determinant size modulo \(k\).
The proofs combine quadratic transformations with a second-order recurrence for an auxiliary family.
As an application, we recover known Hankel evaluations for the square of the weighted Motzkin generating function.
\end{abstract}
\maketitle

\section{Introduction}

For a formal power series \(F(x)=\sum_{j\ge0}f_jx^j\), define its Hankel determinants by
\begin{equation*}
 H_N(F)=\det(f_{i+j})_{0\le i,j<N},\qquad H_0(F)=1.
\end{equation*}
Classical examples include Catalan~\cite{Kratt2010,MW2000}, Motzkin~\cite{Aigner1998}, and Schr\"oder numbers~\cite{BK2005}; broader treatments are given by Aigner~\cite{Aigner1999} and Elouafi~\cite{Elouafi2015}.
Standard methods use nonintersecting paths~\cite{GV1985} and continued fractions~\cite{JT1980}; see Krattenthaler~\cite{Kratt1999} for a survey.
For quadratic generating functions, relevant methods include the direct approach of Chang, Hu, and Zhang~\cite{CHZ2013}, the bivariate method of Gessel and Xin~\cite{GX}, and the quadratic transformation of Sulanke and Xin~\cite{SX}.
The last method was used by Xin~\cite{Xin2009} for the Somos-4 Hankel determinant conjecture and underlies the shifted periodic continued fractions of Wang, Xin, and Zhai~\cite{WXZ}.
Related evaluations exhibit periodicity~\cite{CEF2022,Cigler2011} and Fibonacci polynomials~\cite{Cigler2018}.
Weighted Motzkin determinants~\cite{Cigler2022,CS} and Motzkin convolution powers~\cite{WZMotzkin} provide further examples.

Let \(G^{2,k}(x)\) be the unique formal power series satisfying
\begin{equation}\label{eq:G-definition}
 G^{2,k}(x)=\frac{a}{1+b_1x+b_2x^2+cx^kG^{2,k}(x)},\qquad k\ge4.
\end{equation}
Let \(M(x)\) count paths that start and end on the horizontal axis and never go below it.
Give the steps \((1,0),(2,0),(k-1,1),(1,-1)\) the weights \(-b_1,-b_2,-ac,1\), respectively, with \(x\) marking horizontal displacement.
Decomposition at the first step and, for paths beginning with an up-step, at the first return gives
\begin{equation*}
 M=1-b_1xM-b_2x^2M-acx^kM^2,
\end{equation*}
so \(G^{2,k}=aM\).
The closest evaluations are those of Wang, Xin, and Zhai~\cite{WXZ}.
Their Theorem~10 gives the specialization \(b_2=0\) after replacing its two weights by \(-b_1\) and \(-ac\) and multiplying the determinant of order \(N\) by \(a^N\).
Their Theorem~21 gives the specialization \(b_1=0\): its parameters \((m,q,\ell)=(0,2,k)\) and weights \((-b_2,-ac)\) satisfy its hypotheses precisely when \(k\ge5\).
Theorem~\ref{thm:large-k} extends this latter evaluation to arbitrary \(b_1\), proving that the addition of length-one horizontal steps leaves the Hankel determinants unchanged.
The boundary case \(k=4\), which is outside that instance of Theorem~21, is evaluated in Theorem~\ref{thm:G4} through a second-order recurrence for the auxiliary family in Proposition~\ref{prop:E}.
Both results retain independent horizontal-step weights and hold as polynomial identities at every parameter specialization.

The weighted Motzkin models in Cigler~\cite{Cigler2022} and Chern and Shi~\cite{CS} use the steps \((1,1)\), \((1,0)\), and \((1,-1)\), with endpoint height and coefficient shift as additional indices.
Their setting includes the square of the weighted Motzkin series, recovered in Section~\ref{subsec:Motzkin} at \(k=4\), \(a=1\), \(b_1=-2u\), \(b_2=u^2-2\), and \(c=-1\).
Wang and Zhang~\cite{WZMotzkin} study convolution powers of the ordinary Motzkin series; their square evaluation is the specialization \(u=1\) of this application.
These recovered identities concern a one-parameter subfamily, whereas \(b_1\) and \(b_2\) are independent in \eqref{eq:G-definition}.

For \(k=4\), the formulas involve the two-variable Fibonacci polynomials
\begin{equation}\label{eq:P-definition}
 P_0(u,v)=1,\qquad P_1(u,v)=u,\qquad
 P_{n+1}(u,v)=uP_n(u,v)+vP_{n-1}(u,v)\quad(n\ge1).
\end{equation}
For \(k\ge5\), they are given by residue classes modulo \(k\).
Zero powers denote the constant polynomial one, also when their base is zero.

\begin{theorem}\label{thm:G4}
Set \(G=G^{2,4}\). For every integer \(n\ge0\),
\begin{align}
 H_{2n}(G)&=(-1)^n a^{2n}(ac)^{n(n-1)}P_n(b_2,ac),
 \label{eq:G4-even}\\
 H_{2n+1}(G)&=a^{2n+1}(ac)^{n^2}P_n(b_2,ac).
 \label{eq:G4-odd}
\end{align}
Both formulas are identities in \(\Z[a,b_1,b_2,c]\).
\end{theorem}

\begin{theorem}\label{thm:large-k}
Let \(k\ge5\) be an integer and set \(G=G^{2,k}\).
For every integer \(n\ge0\),
\begin{align}
 H_{kn}(G)
 &=(-1)^{n\binom{k}{2}}a^{kn}(ac)^{n(kn-1)},
 \label{eq:large-0}\\
 H_{kn+1}(G)
 &=(-1)^{n\binom{k}{2}}a^{kn+1}(ac)^{n(kn+1)},
 \label{eq:large-1}\\
 H_{kn+2}(G)
 &=(-1)^{1+n\binom{k}{2}}(n+1)b_2a^{kn+2}(ac)^{n(kn+3)},
 \label{eq:large-2}\\
 H_{kn+j}(G)&=0\qquad(3\le j\le k-2),
 \label{eq:large-zero}\\
 H_{kn+k-1}(G)
 &=(-1)^{(n+1)\binom{k}{2}}(n+1)b_2
 a^{kn+k-1}(ac)^{(n+1)(k(n+1)-3)}.
 \label{eq:large-last}
\end{align}
All five formulas are identities in \(\Z[a,b_1,b_2,c]\).
\end{theorem}

These polynomial identities hold at every parameter value, including specializations at which an earlier determinant vanishes.
In particular, the determinants are independent of \(b_1\), although \([x]G^{2,k}=-ab_1\).

\begin{example}
Set \(q=ac\). The low-order Hankel determinants for \(k=4,5,6\) are as follows.
\begin{center}
\renewcommand{\arraystretch}{1.25}
\begin{tabular}{c|ccc}
\(N\)&\(H_N(G^{2,4})\)&\(H_N(G^{2,5})\)&\(H_N(G^{2,6})\)\\
\hline
0&\(1\)&\(1\)&\(1\)\\
1&\(a\)&\(a\)&\(a\)\\
2&\(-a^2b_2\)&\(-a^2b_2\)&\(-a^2b_2\)\\
3&\(a^3qb_2\)&\(0\)&\(0\)\\
4&\(a^4q^2(b_2^2+q)\)&\(a^4q^2b_2\)&\(0\)\\
5&\(a^5q^4(b_2^2+q)\)&\(a^5q^4\)&\(-a^5q^3b_2\)\\
6&\(-a^6q^6b_2(b_2^2+2q)\)&\(a^6q^6\)&\(-a^6q^5\)
\end{tabular}
\end{center}
\end{example}

\begin{remark}
The independence from \(b_1\) follows from the structure of the reductions.
For \(k=4\), the scalar factor in~\eqref{eq:G4-to-E} does not contain \(b_1\).
In the auxiliary family, the parameter \(b=b_1\) is unchanged by~\eqref{eq:E-parameter-map}; the updates of the other three parameters and the scalar factor in~\eqref{eq:E-two-Hankel} are independent of \(b\).
The terminal values \(H_0(E)=1\) and \(H_1(E)=E(0)\) are also independent of \(b\), so induction on the determinant size proves the assertion.
For \(k\ge5\), both the multiplier in~\eqref{eq:cycle-combined} and all the initial values in Lemma~\ref{lem:cycle-recurrence} are independent of \(b_1\); the cycle recurrence therefore gives the same conclusion.
These arguments take place over the fraction field, and the polynomial specialization argument in Section~\ref{sec:preliminaries} extends the independence to every parameter value.
\end{remark}

Section~\ref{sec:preliminaries} establishes the reduction and specialization argument, and Section~\ref{sec:E} evaluates the auxiliary family needed for \(k=4\).
Sections~\ref{sec:k4} and~\ref{sec:large-k} prove the two theorems; Section~\ref{subsec:Motzkin} recovers the known evaluation for the square of the weighted Motzkin series.

\section{Quadratic reductions and parameter specialization}\label{sec:preliminaries}

All generating functions are formal power series, and \([x^j]F\) denotes the coefficient of \(x^j\).
The order of a nonzero series is its least exponent with nonzero coefficient.
Empty products are one.
We use \(H_N(tF)=t^NH_N(F)\) for constant scalars \(t\).

The one-step Hankel reduction used below is classical and corresponds to a single step of a continued-fraction expansion.
The continued-fraction background includes \(P\)-fraction expansions~\cite{Magnus1962,Roblet1994,Buslaev2010}.
Related determinant reductions occur in work on reciprocal power series~\cite{AndrewsWimp2002}, quadratic transformations~\cite{SX}, and continued fractions and structured matrices~\cite{HoltzTyaglov2012,Han2016}.
We use the normalized reciprocal form of this reduction; for leading coefficient one, it is precisely Han's identity~\cite[Lemma~2.2]{Han2016}.
The following statement allows an arbitrary nonzero leading coefficient; its proof uses the multiplication rule of Gessel and Xin~\cite{GX}.

\begin{proposition}\label{prop:reciprocal-reduction}
Let \(F=f_dx^d+O(x^{d+1})\) be a formal power series over a field, where \(d\ge0\) is an integer and \(f_d\ne0\).
Set
\begin{equation}\label{eq:reciprocal-data}
 w=\frac{f_dx^d}{F},\qquad
 w_L=\sum_{j=0}^{d+1}[x^j]w\,x^j,\qquad
 G=\frac{w_L-w}{x^{d+2}}.
\end{equation}
Then
\begin{equation}\label{eq:reciprocal-reduction}
 H_N(F)=f_d^N(-1)^{\binom{d+1}{2}}H_{N-d-1}(G)
 \qquad(N\ge d+1),
\end{equation}
and \(H_N(F)=0\) for \(1\le N\le d\).
\end{proposition}

\begin{proof}
The definitions in~\eqref{eq:reciprocal-data} give
\[
 \frac{w(x)w(y)}{f_d}\frac{xF(x)-yF(y)}{x-y}
 =\frac{x^{d+1}w_L(y)-y^{d+1}w_L(x)}{x-y}
 +(xy)^{d+1}\frac{xG(x)-yG(y)}{x-y}.
\]
The coefficient of \(x^iy^j\) in \((xF(x)-yF(y))/(x-y)\) is \(f_{i+j}\).
Write \(w_j=[x^j]w\) for \(j\ge0\), and set \(w_j=0\) for \(j<0\).
For \(N\ge d+1\), let
\[
 T_N=(w_{i-j})_{0\le i,j<N},\qquad
 \mathcal H_N(F)=(f_{i+j})_{0\le i,j<N}.
\]
Multiplication by \(w(x)\) replaces row \(i\) by the sum of rows \(r\), weighted by \(w_{i-r}\), for \(0\le r\le i\); multiplication by \(w(y)\) performs the corresponding operation on columns.
Thus the leading \(N\times N\) coefficient matrix on the left is
\[
 f_d^{-1}T_N\mathcal H_N(F)T_N^{\mathsf T}.
\]
Since \(w_0=1\), the matrix \(T_N\) is lower triangular with diagonal entries one, and its transpose is upper triangular with the same diagonal.
Hence the determinant of the coefficient matrix on the left is \(f_d^{-N}H_N(F)\).

Write \(G=\sum_{j\ge0}g_jx^j\).
The first summand on the right is supported in \(0\le i,j\le d\), where its coefficient of \(x^iy^j\) is \(w_{i+j-d}\).
The second summand has both exponents at least \(d+1\).
Consequently the coefficient matrix on the right has the block form
\[
 \begin{pmatrix}
 A_d&0\\
 0&\mathcal H_{N-d-1}(G)
 \end{pmatrix},\qquad
 A_d=(w_{i+j-d})_{0\le i,j\le d},\qquad
 \mathcal H_{N-d-1}(G)=(g_{i+j})_{0\le i,j<N-d-1}.
\]
Reversing the \(d+1\) columns of \(A_d\) gives the lower triangular matrix
\((w_{i-j})_{0\le i,j\le d}\), whose diagonal entries are one.
Column reversal has sign \((-1)^{\binom{d+1}{2}}\), so
\(\det A_d=(-1)^{\binom{d+1}{2}}\).
When \(N=d+1\), the second block is empty and has determinant one.
Taking determinants proves \eqref{eq:reciprocal-reduction}.
For \(1\le N\le d\), the first row of the Hankel matrix of \(F\) is zero.
\end{proof}

For polynomial specialization, let \(R=\Z[t_1,\ldots,t_s]\), where \(t_1,\ldots,t_s\) are independent parameters, and consider an equation
\begin{equation}\label{eq:formal-general}
 D(x)F(x)+Q(x)F(x)^2=A(x),\qquad D(0)=1,\quad Q(0)=0.
\end{equation}
Assume \(A,D,Q\in R[[x]]\).
The parameter list is \((a,b_1,b_2,c)\) for \(G^{2,k}\), and \((a,b,c,d)\) for the auxiliary family \(E\) in~\eqref{eq:E}.
Coefficient comparison determines a unique \(F\in R[[x]]\) recursively, without division by a parameter.
Thus \(H_N(F)\in R\) for every \(N\ge0\).
We perform rational transformations over \(\Q(t_1,\ldots,t_s)\).
When both sides of the final equality belong to \(R\), injectivity of \(R\) into this field makes the equality a polynomial identity.
It therefore holds under every parameter specialization, even if an intermediate rational transform is undefined.

In particular, Theorems~\ref{thm:G4} and~\ref{thm:large-k} include the specializations with \(ac=0\).
For every \(k\ge4\), if \(a=0\), all determinants of positive order vanish; if \(c=0\), then
\[
 H_1=a,\qquad H_2=-a^2b_2,\qquad H_N=0\quad(N\ge3).
\]

\section{The auxiliary four-parameter family}\label{sec:E}

Let \(E(x)\) be the unique formal power series satisfying
\begin{equation}\label{eq:E}
 E(x)=\frac{a}{1+bx+cx^2+x^2(dx^2-bx-1)E(x)}.
\end{equation}
Define polynomials \(B_n\) by
\begin{equation}\label{eq:B-recurrence}
 B_0=1,\qquad B_1=c-a,\qquad
 B_{n+1}=(c-2a)B_n+a(c+d-a)B_{n-1}\quad(n\ge1).
\end{equation}

\begin{proposition}\label{prop:E}
For every integer \(n\ge0\),
\begin{align}
 H_{2n}(E)
 &=(-1)^n a^{n(n+1)}(c+d-a)^{n(n-1)}B_n,
 \label{eq:E-even}\\
 H_{2n+1}(E)
 &=a^{(n+1)^2}(c+d-a)^{n^2}B_n.
 \label{eq:E-odd}
\end{align}
Both formulas are identities in \(\Z[a,b,c,d]\) and hold at every parameter specialization.
\end{proposition}

We first work over \(\Q(a,b,c,d)\), with independent parameters.

\begin{lemma}\label{lem:E-two-steps}
The first normalized reduction of \(E\) is \(F=(1+bx-dx^2)E-c\), which satisfies
\begin{equation}\label{eq:E-F-equation}
 F=\frac{(a-c)(1+bx)-ad x^2}{1+bx-cx^2-x^2F}.
\end{equation}
Its next normalized reduction is
\begin{equation}\label{eq:E-second-series}
 E'=\frac{1-(a-c)/F}{x^2}.
\end{equation}
This series satisfies \eqref{eq:E} with \(b\) unchanged and
\begin{equation}\label{eq:E-parameter-map}
 a'=a-\frac{ad}{a-c},\qquad
 c'=c-\frac{2ad}{a-c},\qquad d'=\frac{ad}{a-c}.
\end{equation}
For \(N\ge2\),
\begin{equation}\label{eq:E-two-Hankel}
 H_N(E)=a^N(a-c)^{N-1}H_{N-2}(E').
\end{equation}
\end{lemma}

\begin{proof}
The truncation of \(a/E\) through degree one is \(1+bx\).
Proposition~\ref{prop:reciprocal-reduction} gives
\[
 F=\frac{1+bx-a/E}{x^2}=(1+bx-dx^2)E-c,
 \qquad H_N(E)=a^NH_{N-1}(F).
\]
Since \(E(1+bx-x^2F)=a\), elimination of \(E\) yields
\[
 (F+c)(1+bx-x^2F)=a(1+bx-dx^2),
\]
which rearranges to \eqref{eq:E-F-equation}.
Its constant and linear coefficients give \(F(0)=a-c\) and \([x]F=0\).
The second reduction is thus \eqref{eq:E-second-series}, with
\(H_{N-1}(F)=(a-c)^{N-1}H_{N-2}(E')\), proving \eqref{eq:E-two-Hankel}.

Substitute \(F=(a-c)/(1-x^2E')\) into \eqref{eq:E-F-equation} and clear denominators:
\begin{align*}
 &(a-c)(1+bx-cx^2)(1-x^2E')-x^2(a-c)^2\\
 &\qquad=\bigl((a-c)(1+bx)-ad x^2\bigr)(1-x^2E')^2.
\end{align*}
Collecting the constant, \(E'\), and \(E'^2\) terms and cancelling \((a-c)x^2\) gives
\[
 E'\left[1+bx+\left(c-\frac{2ad}{a-c}\right)x^2
 +x^2\left(\frac{ad}{a-c}x^2-bx-1\right)E'\right]
 =a-\frac{ad}{a-c}.
\]
Cancellation is valid in \(\Q(a,b,c,d)[[x]]\), where \(a-c\ne0\).
This proves \eqref{eq:E-parameter-map}.
\end{proof}

\begin{proof}[Proof of Proposition~\ref{prop:E}]
Equation~\eqref{eq:E} has the form \eqref{eq:formal-general}, so both sides of \eqref{eq:E-even} and~\eqref{eq:E-odd} belong to \(\Z[a,b,c,d]\).
We prove these polynomial identities together by induction on \(n\).
For \(n=0\), they are \(H_0(E)=1\) and \(H_1(E)=a\).
Assume they hold for \(n-1\), where \(n\ge1\), and work in \(\Q(a,b,c,d)\).

Let \(B_j'\) denote \(B_j\) evaluated at the parameters \((a',c',d')\) in \eqref{eq:E-parameter-map}.
That map gives
\begin{align*}
 c'+d'-a'&=c-a,& a'&=\frac{a(c+d-a)}{c-a},\\
 c'-2a'&=c-2a,& a'(c'+d'-a')&=a(c+d-a).
\end{align*}
Thus the recurrence coefficients in \eqref{eq:B-recurrence} are unchanged.
Also, \(B_0'=1=B_1/(c-a)\), and
\[
 B_1'=c'-a'=c-a+\frac{ad}{c-a}
 =\frac{(c-a)^2+ad}{c-a}=\frac{B_2}{c-a}.
\]
The common recurrence and these initial values imply
\begin{equation}\label{eq:B-shift}
 B_j'=\frac{B_{j+1}}{c-a}\qquad(j\ge0).
\end{equation}

The induction hypothesis consists of polynomial identities, so it may be evaluated at \((a',b,c',d')\).
It gives
\begin{align*}
 H_{2n-2}(E')
 &=(-1)^{n-1}(a')^{n(n-1)}(c'+d'-a')^{(n-1)(n-2)}B_{n-1}',\\
 H_{2n-1}(E')
 &=(a')^{n^2}(c'+d'-a')^{(n-1)^2}B_{n-1}'.
\end{align*}
Substitution of \(c'+d'-a'=c-a\) and \(B_{n-1}'=B_n/(c-a)\) from~\eqref{eq:B-shift} into these formulas, followed by~\eqref{eq:E-two-Hankel}, yields
\begin{align*}
 H_{2n}(E)
 &=(-1)^{n-1}a^{2n}(a-c)^{2n-1}
 (a')^{n(n-1)}(c-a)^{(n-1)(n-2)}\frac{B_n}{c-a}\\
 &=(-1)^n a^{n(n+1)}(c+d-a)^{n(n-1)}B_n,\\
 H_{2n+1}(E)
 &=a^{2n+1}(a-c)^{2n}(a')^{n^2}(c-a)^{(n-1)^2}\frac{B_n}{c-a}\\
 &=a^{(n+1)^2}(c+d-a)^{n^2}B_n.
\end{align*}
Indeed, after substituting for \(a'\), the exponents of \(c-a\) are
\[
 2n-2-n(n-1)+(n-1)(n-2)=0,
 \qquad 2n-1-n^2+(n-1)^2=0,
\]
and the even-order sign is \((-1)^{3n-2}=(-1)^n\).
These equalities in the fraction field are polynomial identities by Section~\ref{sec:preliminaries}, which completes the induction.
\end{proof}

For example, at \(c=a\), recurrence~\eqref{eq:B-recurrence} gives \(B_1=0\), \(B_2=ad\), and hence
\[
 H_2(E)=H_3(E)=0,\qquad H_4(E)=a^7d^3,\qquad H_5(E)=a^{10}d^5.
\]
Thus for \(ad\ne0\), zero determinants are followed by nonzero ones.

\section{Two horizontal steps: the case \texorpdfstring{$k=4$}{k=4}}\label{sec:k4}

\subsection{Proof of the determinant formulas}
\begin{proof}[Proof of Theorem~\ref{thm:G4}]
Work over \(\Q(a,b_1,b_2,c)\), and set \(G=G^{2,4}\), \(F=G/a\), so
\[
 \frac1F=1+b_1x+b_2x^2+acx^4F,\qquad F(0)=1.
\]
The first reduction in Proposition~\ref{prop:reciprocal-reduction} gives
\[
 G_1=-b_2-acx^2F,\qquad H_N(G)=a^NH_{N-1}(G_1)\quad(N\ge1).
\]
Since \(G_1(0)=-b_2\) and \([x]G_1=0\), the truncation of \(-b_2/G_1\) through degree one equals \(1\), and
\begin{equation}\label{eq:G4-second}
 G_2=\frac{1+b_2/G_1}{x^2}=\frac{acF}{b_2+acx^2F}.
\end{equation}
The two reductions yield
\begin{equation}\label{eq:G4-to-E}
 H_N(G)=a^N(-b_2)^{N-1}H_{N-2}(G_2)\quad(N\ge2).
\end{equation}
Solving \eqref{eq:G4-second} gives \(F=b_2G_2/(ac(1-x^2G_2))\).
Substitution in \(F(1+b_1x+b_2x^2)+acx^4F^2=1\) and multiplication by \(ac(1-x^2G_2)^2\) give
\begin{align*}
 &b_2(1+b_1x+b_2x^2)G_2(1-x^2G_2)+b_2^2x^4G_2^2\\
 &\qquad=ac(1-2x^2G_2+x^4G_2^2).
\end{align*}
Collecting terms and dividing by \(b_2\), we obtain
\[
 G_2=\frac{ac/b_2}{1+b_1x+(b_2+2ac/b_2)x^2
 +x^2(-(ac/b_2)x^2-b_1x-1)G_2}.
\]
To distinguish the parameters of the auxiliary family from those of \(G\), write its defining equation as
\[
 E=\frac{\alpha}{1+\beta x+\gamma x^2+x^2(\delta x^2-\beta x-1)E}.
\]
Thus \(G_2\) is this series at the parameter values
\[
 (\alpha,\beta,\gamma,\delta)
 =\left(\frac{ac}{b_2},\ b_1,\ b_2+\frac{2ac}{b_2},\ -\frac{ac}{b_2}\right).
\]
In particular,
\[
 \gamma+\delta-\alpha=b_2,\qquad
 \gamma-2\alpha=b_2,\qquad
 \alpha(\gamma+\delta-\alpha)=ac.
\]
Writing \(B_j^*\) for the corresponding value of \(B_j\), we have
\[
 B_0^*=1,\qquad B_1^*=b_2+ac/b_2,\qquad
 B_{j+1}^*=b_2B_j^*+acB_{j-1}^*.
\]
Consequently
\begin{equation}\label{eq:Bstar-P}
 b_2B_{n-1}^*=P_n(b_2,ac)\qquad(n\ge1):
\end{equation}
the first two values are \(b_2\) and \(b_2^2+ac\), and the recurrences agree.

For \(n\ge1\), Proposition~\ref{prop:E} gives
\begin{align*}
 H_{2n-2}(G_2)
 &=(-1)^{n-1}\alpha^{n(n-1)}
 (\gamma+\delta-\alpha)^{(n-1)(n-2)}B_{n-1}^*\\
 &=(-1)^{n-1}\left(\frac{ac}{b_2}\right)^{n(n-1)}
 b_2^{(n-1)(n-2)}B_{n-1}^*,\\
 H_{2n-1}(G_2)
 &=\alpha^{n^2}(\gamma+\delta-\alpha)^{(n-1)^2}B_{n-1}^*\\
 &=\left(\frac{ac}{b_2}\right)^{n^2}b_2^{(n-1)^2}B_{n-1}^*.
\end{align*}
Substitution into \eqref{eq:G4-to-E} now yields
\begin{align*}
 H_{2n}(G)&=a^{2n}(-b_2)^{2n-1}H_{2n-2}(G_2)\\
 &=(-1)^n a^{2n}(ac)^{n(n-1)}b_2B_{n-1}^*,\\
 H_{2n+1}(G)&=a^{2n+1}(-b_2)^{2n}H_{2n-1}(G_2)\\
 &=a^{2n+1}(ac)^{n^2}b_2B_{n-1}^*.
\end{align*}
The exponents of \(b_2\) after substitution are, respectively,
\[
 2n-1-n(n-1)+(n-1)(n-2)=1,
 \qquad 2n-n^2+(n-1)^2=1.
\]
Equation~\eqref{eq:Bstar-P} proves \eqref{eq:G4-even} and~\eqref{eq:G4-odd}; for \(n=0\), these are \(H_0(G)=1\) and \(H_1(G)=a\).
By \eqref{eq:G-definition}, \eqref{eq:P-definition}, and Section~\ref{sec:preliminaries}, both sides are polynomials, so the formulas hold at every parameter value.
\end{proof}

\subsection{Consequences for vanishing determinants}

The pattern of vanishing determinants follows from Theorem~\ref{thm:G4} and the recurrence for \(P_n\).
The order of a root of unity is the least positive integer \(m\) for which its \(m\)-th power equals \(1\).

\begin{corollary}\label{cor:G4-zeros}
Let \(a,b_1,b_2,c\in\C\) with \(ac\ne0\), and let \(r_+,r_-\) be the roots, counted with multiplicity, of
\begin{equation*}
 z^2-b_2z-ac=0.
\end{equation*}
If the roots are distinct and \(r_+/r_-\) is a root of unity of order \(m\), then \(m\ge2\) and
\begin{equation}\label{eq:zero-residues}
 H_N(G^{2,4})=0
 \quad\Longleftrightarrow\quad
 N\equiv2m-2\ \text{or}\ 2m-1\pmod{2m}.
\end{equation}
In all other cases, every \(H_N(G^{2,4})\) is nonzero.

If \(b_2\ne0\) and \(-b_2^2/(ac)\in\Q\), a zero determinant occurs exactly when \(-b_2^2/(ac)\in\{1,2,3\}\).
The corresponding zero indices are \(N\equiv4,5\pmod6\), \(N\equiv6,7\pmod8\), and \(N\equiv10,11\pmod{12}\), respectively.
For \(b_2=0\), the zero indices are \(N\equiv2,3\pmod4\).
\end{corollary}

\begin{proof}
By Theorem~\ref{thm:G4}, the determinants of orders \(2n\) and \(2n+1\) vanish exactly when \(P_n(b_2,ac)=0\).
Since \(r_+r_-=-ac\ne0\), both roots are nonzero.
The recurrence \eqref{eq:P-definition} gives
\[
 P_n(b_2,ac)=
 \begin{cases}
 (r_+^{n+1}-r_-^{n+1})/(r_+-r_-),&r_+\ne r_-,\\
 (n+1)r_+^n,&r_+=r_-.
 \end{cases}
\]
Both expressions have initial values \(1,b_2\) and satisfy the defining recurrence.
The repeated-root expression is nonzero.
For distinct roots, vanishing is equivalent to \((r_+/r_-)^{n+1}=1\).
If the ratio has order \(m\), then \(m\ge2\) and \(n\equiv m-1\pmod m\); substituting \(N=2n,2n+1\) gives \eqref{eq:zero-residues}.
A ratio that is not a root of unity gives no zeros.

Suppose now that \(b_2\ne0\) and \(-b_2^2/(ac)\in\Q\).
If a determinant vanishes, then \(\zeta=r_+/r_-\) is a root of unity and
\begin{equation}\label{eq:rational-root-ratio}
 -\frac{b_2^2}{ac}=2+\zeta+\zeta^{-1}.
\end{equation}
If \(m\) is the order of \(\zeta\), then both \(\zeta\) and \(\zeta^{-1}\) satisfy the monic polynomial \(t^m-1\), so they are algebraic integers.
Their sum with \(2\) is also an algebraic integer, and it is rational by the assumption on \(-b_2^2/(ac)\).
A rational algebraic integer is an integer: if \(p/q\), in lowest terms with \(q>0\), satisfies a monic polynomial of degree \(s\) with integer coefficients, clearing denominators shows that \(q\) divides \(p^s\), hence \(q=1\).
Moreover, \(|\zeta|=1\) gives \(\zeta^{-1}=\overline\zeta\), so
\[
 2+\zeta+\zeta^{-1}=2+2\operatorname{Re}\zeta\in[0,4].
\]
The values zero and four give \(b_2=0\) and a repeated root, respectively, so only \(1,2,3\) remain.
Conversely, these three values give distinct roots, and \eqref{eq:rational-root-ratio} yields, respectively,
\[
 \zeta^2+\zeta+1=0,\qquad \zeta^2+1=0,\qquad \zeta^2-\zeta+1=0.
\]
Thus the ratio has order \(3,4,6\), giving the stated indices by \eqref{eq:zero-residues}.
Finally, for \(b_2=0\) the ratio is \(-1\), of order two, which gives the remaining case.
\end{proof}

\subsection{The square of the weighted Motzkin series}\label{subsec:Motzkin}

A Motzkin path starts at \((0,0)\), uses the steps \((1,1)\), \((1,0)\), and \((1,-1)\), and ends on the horizontal axis without going below it.
Let \(M_u(x)\), the weighted Motzkin series, be the generating function of these paths, including the empty path, with \(x\) marking the number of steps and step weights \(1,u,1\), respectively.
Decomposition at the first step and at the first return gives \(M_u=1+uxM_u+x^2M_u^2\).
The ordinary Motzkin series is \(M_1(x)\).
The following known evaluation is contained in Chern and Shi~\cite[Theorems~1.2 and~1.4]{CS}.
In Theorem~1.4, take endpoint height and coefficient shift both equal to \(1\): nonnegative paths ending at height \(1\) have generating function \(xM_u^2\), and the coefficient shift gives \(M_u^2\).
The specialization \(u=1\) is also recorded by Wang and Zhang~\cite[Theorem~1]{WZMotzkin}.

\begin{corollary}\label{cor:Motzkin-square}
For every \(n\ge0\),
\begin{equation}\label{eq:Motzkin-square-H}
 H_{2n}(M_u^2)=H_{2n+1}(M_u^2)=(-1)^nP_n(u^2-2,-1).
\end{equation}
At \(u=1\), the determinant sequence beginning at \(H_0\) has period twelve, with repeating block
\[
 1,1,1,1,0,0,-1,-1,-1,-1,0,0.
\]
\end{corollary}

\begin{proof}
Squaring \(M_u(1-ux)=1+x^2M_u^2\) and rearranging gives
\[
 M_u^2\bigl(1-2ux+(u^2-2)x^2-x^4M_u^2\bigr)=1.
\]
Thus \(M_u^2\) is \(G^{2,4}\) with \(a=1\), \(b_1=-2u\), \(b_2=u^2-2\), and \(c=-1\).
Theorem~\ref{thm:G4} gives \eqref{eq:Motzkin-square-H}, since \(n(n-1)\) is even and \(n^2\equiv n\pmod2\).
For \(u=1\), the recurrence \eqref{eq:P-definition} gives \(P_n(-1,-1)\) the repeating block \(1,-1,0\).
Multiplication by \((-1)^n\) gives \((-1)^nP_n(-1,-1)\) the repeating block
\[
 1,1,0,-1,-1,0.
\]
Repeating each value at sizes \(2n,2n+1\) gives the stated twelve-term block.
\end{proof}

\section{Two horizontal steps: the case \texorpdfstring{$k\ge5$}{k>=5}}\label{sec:large-k}

\subsection{A cycle of four transformations}
Fix \(k\ge5\), set \(G=G^{2,k}\), and work over \(\Q(a,b_1,b_2,c)\).
Set \(F=G/a\).
Then
\begin{equation*}
 \frac1F=1+b_1x+b_2x^2+acx^kF,\qquad F(0)=1.
\end{equation*}
The truncation of its reciprocal through degree one is \(1+b_1x\).
Its first reduction is consequently \(-b_2-acx^{k-2}F\), and
\begin{equation}\label{eq:large-first-reduction}
 H_N(G)=a^N H_{N-1}(G_1^{(1)})\qquad(N\ge1),
\end{equation}
where the following four series describe every later cycle.

\begin{lemma}\label{lem:four-cycle}
For each positive integer \(p\), define
\begin{align}
 G_1^{(p)}&=-pb_2-acx^{k-2}F,\quad
 G_2^{(p)}=\frac{acx^{k-4}F}{pb_2+acx^{k-2}F},
 \label{eq:cycle-2}\\
 G_3^{(p)}&=-\frac{ac}{pb_2}-acx^2F, \quad
 G_4^{(p)}=\frac{pb_2F}{1+pb_2x^2F}.
 \label{eq:cycle-4}
\end{align}
Successive normalized reductions from Proposition~\ref{prop:reciprocal-reduction} take
\[
 G_1^{(p)}\longmapsto G_2^{(p)}\longmapsto
 G_3^{(p)}\longmapsto G_4^{(p)}\longmapsto G_1^{(p+1)}.
\]
Here each arrow includes the indicated leading-coefficient normalization.
The corresponding determinant identities are
\begin{align}
 H_N(G_1^{(p)})
 &=(-pb_2)^N H_{N-1}(G_2^{(p)}) &&(N\ge1),
 \label{eq:cycle-H1}\\
 H_N(G_2^{(p)})
 &=(-1)^{\binom{k-3}{2}}\left(\frac{ac}{pb_2}\right)^N
 H_{N-k+3}(G_3^{(p)}) &&(N\ge k-3),
 \label{eq:cycle-H2}\\
 H_N(G_3^{(p)})
 &=\left(-\frac{ac}{pb_2}\right)^N H_{N-1}(G_4^{(p)}) &&(N\ge1),
 \label{eq:cycle-H3}\\
 H_N(G_4^{(p)})
 &=(pb_2)^N H_{N-1}(G_1^{(p+1)}) &&(N\ge1).
 \label{eq:cycle-H4}
\end{align}
\end{lemma}

\begin{proof}
The denominators in \eqref{eq:cycle-2} and~\eqref{eq:cycle-4} have nonzero constant terms in the chosen field.
The orders of the four series are \(0,k-4,0,0\), and their leading coefficients are
\[
 -pb_2,\qquad ac/(pb_2),\qquad -ac/(pb_2),\qquad pb_2.
\]
For the first and third series, the normalized reciprocals are, respectively,
\[
 \frac{-pb_2}{G_1^{(p)}}=\frac{pb_2}{pb_2+acx^{k-2}F},
 \qquad
 \frac{-ac/(pb_2)}{G_3^{(p)}}=\frac1{1+pb_2x^2F}.
\]
Both have truncation one through degree one; the quotients \((1-w)/x^2\) are precisely \(G_2^{(p)}\) and \(G_4^{(p)}\).
Proposition~\ref{prop:reciprocal-reduction} gives \eqref{eq:cycle-H1} and~\eqref{eq:cycle-H3}.

For \(G_2^{(p)}\), the normalized reciprocal is
\[
 w=\frac{(ac/(pb_2))x^{k-4}}{G_2^{(p)}}
 =1+b_1x+b_2x^2+\frac{ac}{pb_2}x^{k-2}+acx^kF.
\]
Its truncation through degree \(k-3\) is \(1+b_1x+b_2x^2\), since \(k\ge5\).
Thus
\[
 \frac{1+b_1x+b_2x^2-w}{x^{k-2}}
 =-\frac{ac}{pb_2}-acx^2F=G_3^{(p)},
\]
which gives \eqref{eq:cycle-H2}, including its sign and order shift.
For \(k=4\), the required degree-one truncation would exclude \(b_2x^2\); this is why that case needs a different reduction.

Finally, for \(G_4^{(p)}\),
\[
 w=\frac{pb_2}{G_4^{(p)}}=\frac1F+pb_2x^2
 =1+b_1x+(p+1)b_2x^2+acx^kF.
\]
Its degree-one truncation is \(1+b_1x\), so its transform is
\((1+b_1x-w)/x^2=G_1^{(p+1)}\).
This proves \eqref{eq:cycle-H4}.
\end{proof}

\subsection{The determinant recurrence and its initial values}

\begin{lemma}\label{lem:cycle-recurrence}
For \(N\ge k\),
\begin{equation}\label{eq:cycle-combined}
 H_N(G_1^{(p)})=(-1)^{\binom{k}{2}}(ac)^{2N-k+1}
 H_{N-k}(G_1^{(p+1)}).
\end{equation}
For the initial indices \(0\le j<k\),
\begin{align}
 H_0(G_1^{(p)})&=1, \quad H_1(G_1^{(p)})=-pb_2,
 \label{eq:cycle-initial01}\\
 H_j(G_1^{(p)})&=0 \, \quad  (2\le j\le k-3),
 \label{eq:cycle-initialzeros}\\
 H_{k-2}(G_1^{(p)})&=(-1)^{\binom{k}{2}}pb_2(ac)^{k-3},
 \label{eq:cycle-initial-k2}\\
 H_{k-1}(G_1^{(p)})&=(-1)^{\binom{k}{2}}(ac)^{k-1}.
 \label{eq:cycle-initial-k1}
\end{align}
\end{lemma}

\begin{proof}
Apply Lemma~\ref{lem:four-cycle} successively at orders
\(N,N-1,N-k+2,N-k+1\); all four lower bounds hold for \(N\ge k\).
At \(N=k\), these orders are \(k,k-1,2,1\), and the terminal determinant is \(H_0(G_1^{(p+1)})=1\), so the recurrence includes its boundary case.
Before collecting powers, the four factors give
\begin{align*}
 H_N(G_1^{(p)})
 &=(-pb_2)^N
 \left[(-1)^{\binom{k-3}{2}}\left(\frac{ac}{pb_2}\right)^{N-1}\right]\\
 &\qquad\times\left(-\frac{ac}{pb_2}\right)^{N-k+2}
 (pb_2)^{N-k+1}H_{N-k}(G_1^{(p+1)}).
\end{align*}
The exponent of \(pb_2\) is \(N-(N-1)-(N-k+2)+(N-k+1)=0\), and that of \(ac\) is \((N-1)+(N-k+2)=2N-k+1\).
The sign exponent satisfies
\[
 N+\binom{k-3}{2}+N-k+2=\binom{k}{2}+2N-4k+8.
\]
The remaining order is \(N-k\), proving \eqref{eq:cycle-combined}.

The values at zero and one follow from \(H_0=1\) and \(G_1^{(p)}(0)=-pb_2\).
For \(2\le j\le k-3\), \eqref{eq:cycle-H1} reduces the determinant to \(H_{j-1}(G_2^{(p)})\).
Since \(G_2^{(p)}\) has order \(k-4\) and \(1\le j-1\le k-4\), its Hankel matrix has a zero first row, proving \eqref{eq:cycle-initialzeros}.
For \(j=k-2,k-1\), the first two reductions give
\[
 H_j(G_1^{(p)})=(-pb_2)^j(-1)^{\binom{k-3}{2}}
 \left(\frac{ac}{pb_2}\right)^{j-1}H_{j-k+2}(G_3^{(p)}).
\]
Insert \(H_0(G_3^{(p)})=1\) and \(H_1(G_3^{(p)})=-ac/(pb_2)\), respectively.
The resulting signs agree with \eqref{eq:cycle-initial-k2} and~\eqref{eq:cycle-initial-k1}, because
\[
 \binom{k}{2}-\binom{k-3}{2}-(k-2)=2k-4,
 \qquad
 \binom{k}{2}-\binom{k-3}{2}-k=2k-6.
\]
The powers of \(pb_2\) are one and zero, and those of \(ac\) are \(k-3\) and \(k-1\), as asserted.
For \(k=5\), the zero range \(2\le j\le k-3\) contains only \(j=2\), while the two subsequent initial values occur at \(j=3,4\); thus these cases cover every initial index without overlap.
\end{proof}

\begin{proof}[Proof of Theorem~\ref{thm:large-k}]
Write \(N=\ell k+r\), where \(\ell\ge0\) and \(0\le r<k\).
Iterating \eqref{eq:cycle-combined} gives
\begin{equation}\label{eq:cycle-iterated}
 H_N(G_1^{(p)})=(-1)^{\ell\binom{k}{2}}(ac)^{\ell(k\ell+2r+1)}
 H_r(G_1^{(p+\ell)}),
\end{equation}
since the accumulated exponent is
\[
 \sum_{j=0}^{\ell-1}\bigl(2(N-jk)-k+1\bigr)
 =\ell(2N-k+1)-k\ell(\ell-1)=\ell(k\ell+2r+1).
\]
For \(\ell=0\), this is the identity with empty product.
We now combine \eqref{eq:large-first-reduction} and~\eqref{eq:cycle-iterated}, starting with \(p=1\).
For every \(N\ge1\), write \(N-1=\ell k+r\), with \(\ell\ge0\) and \(0\le r<k\).
Then
\[
 H_N(G)=a^N(-1)^{\ell\binom{k}{2}}
 (ac)^{\ell(k\ell+2r+1)}H_r(G_1^{(\ell+1)}).
\]

For \(H_{kn}(G)\), \(n\ge1\), the reduced order is \(kn-1=(n-1)k+(k-1)\).
The terminal value is given by \eqref{eq:cycle-initial-k1}, so the sign is \((-1)^{n\binom{k}{2}}\) and the exponent of \(ac\) is
\[
 (n-1)\bigl(k(n-1)+2(k-1)+1\bigr)+(k-1)=n(kn-1).
\]
This proves \eqref{eq:large-0}; its case \(n=0\) is \(H_0(G)=1\).
For \(H_{kn+1}(G)\) and \(H_{kn+2}(G)\), use \((\ell,r)=(n,0),(n,1)\), respectively.
By \eqref{eq:cycle-initial01}, the terminal values are \(1\) and \(-(n+1)b_2\), with exponents \(n(kn+1)\) and \(n(kn+3)\); these give \eqref{eq:large-1} and~\eqref{eq:large-2}.
For \(H_{kn+j}(G)\), \(3\le j\le k-2\), use \((\ell,r)=(n,j-1)\).
The zero initial values in \eqref{eq:cycle-initialzeros} give \eqref{eq:large-zero}.

Finally, for \(H_{kn+k-1}(G)\), use \((\ell,r)=(n,k-2)\).
By \eqref{eq:cycle-initial-k2}, the terminal value is
\((-1)^{\binom{k}{2}}(n+1)b_2(ac)^{k-3}\).
The total exponent and sign are
\[
 n(kn+2k-3)+k-3=(n+1)(k(n+1)-3),\qquad
 (-1)^{(n+1)\binom{k}{2}},
\]
proving \eqref{eq:large-last}.
The coefficients of \(G\) are polynomials by \eqref{eq:G-definition} and Section~\ref{sec:preliminaries}.
All exponents in the final formulas are nonnegative in the stated ranges, so both sides are polynomials.
The same specialization argument proves the formulas at every parameter value.
\end{proof}

For complex parameters, Theorem~\ref{thm:large-k} also gives the zero indices directly.
If \(acb_2\ne0\), then \(H_N(G^{2,k})=0\) exactly when \(N\bmod k\in\{3,\ldots,k-2\}\).
If \(ac\ne0\) and \(b_2=0\), the nonzero determinants are exactly those with \(N\equiv0,1\pmod k\).
The specializations with \(ac=0\) are recorded in Section~\ref{sec:preliminaries}.

For the family~\eqref{eq:G-definition}, Theorems~\ref{thm:G4} and~\ref{thm:large-k} extend the single-horizontal-step specializations identified above to independent weights for steps of lengths one and two.
The auxiliary recurrence yields the Fibonacci-polynomial formulas at \(k=4\), while the four-step cycle yields the residue-class formulas at \(k\ge5\).
The weighted Motzkin-square evaluation in Corollary~\ref{cor:Motzkin-square} is a known specialization of this family.

\raggedbottom
\subsection*{Use of Generative-AI tools declaration}
The main research of this paper was completed two years before the present revision, with no AI used at that time.
Before submission, ChatGPT was used to check technical details and polish the structure and exposition of the manuscript. All mathematical arguments, computations, proofs and results have been independently verified by the authors. The authors take full responsibility for the content of this manuscript.

\subsection*{Acknowledgments}
The authors thank the referees for their helpful comments.
Ying Wang was partially supported by the Natural Science Foundation of Henan Province (Grant No. 262300422649).

\bibliographystyle{amsplain}
\providecommand{\bysame}{\leavevmode\hbox to3em{\hrulefill}\thinspace}
\providecommand{\MR}{\relax\ifhmode\unskip\space\fi MR }
\providecommand{\MRhref}[2]{%
  \href{http://www.ams.org/mathscinet-getitem?mr=#1}{#2}
}
\providecommand{\href}[2]{#2}

\end{document}